\documentclass[11pt, english, draft]{article}
\usepackage{amssymb,amsmath}
\title{\bf  OD-Characterization of Some Simple Unitary Groups}
\author{ {\bf Majid Akbari}, {\bf Xiaoyou Chen}\thanks{The second author was supported by Natural Key Fund of Education Department of Henan Province (No.17A110004),
Natural Funds of Henan Province (No.182102410049), Fund of Foreign Experts Affairs of Henan Province,
and National Natural Science Foundation of China (Nos.11571129, 11771356).} ,
{\bf Faisal Hassani} and {\bf Ali Reza Moghaddamfar}}

\newenvironment{proof}{\noindent {\em {Proof}}.}{$\square$
\medskip}
\newtheorem{theorem}{Theorem}[section]
\newtheorem{definition}[theorem]{Definition}
\newtheorem{corollary}[theorem]{Corollary}

\newtheorem{pro}[theorem]{Proposition}

\newtheorem{lm}[theorem]{Lemma}

\begin{document}
\newcommand{\f}{\frac}
\newcommand{\sta}{\stackrel}
\maketitle
\begin{abstract}
\noindent The degree pattern of a finite group is the degree sequence of its prime graph in ascending order of vertices. We say that the problem of OD-characterization is solved for a finite group if we determine  the number of pairwise nonisomorphic finite groups with the same order and degree pattern as the group under consideration. In this article the problem of OD-characterization is solved for some simple unitary groups.   It was shown, in particular,  that the simple unitary groups $U_3(q)$ and $U_4(q)$ are OD-characterizable, where $q$ is a prime power $<10^2$.
 \end{abstract}
\renewcommand{\baselinestretch}{1.1}
\def\thefootnote{ \ }
\footnotetext{{\em AMS subject Classification {\rm 2010}}:
20D05, 20D06, 20D08.\\[0.1cm]
}
\section{\sc Introduction}
Throughout this article, all the groups under consideration are
{\em finite}, and simple groups are {\em nonabelian}. Given a
group $G$, the {\em spectrum} $\omega(G)$ of $G$ is the set of orders of
elements in $G$. Clearly, the spectrum $\omega(G)$ is {\em closed} and
{\em partially ordered} by the divisibility relation, and hence it is
uniquely determined by the set $\mu(G)$ of its elements which are
{\em maximal} under the divisibility relation.

One of the most well-known graphs associated with a group  $G$ is the
{\em prime graph} (or {\em Gruenberg-Kegel graph}) denoted by ${\rm GK}(G)$.
In this graph, the vertex set is $\pi(G)$, the set of all prime divisors of the order of $G$, and two distinct
vertices $p$ and $q$ are joined by an edge (written by $p\sim
q$) if and only if $G$ contains an element of order  $pq$.  If $p_1<p_2<\cdots<p_k$ are all
vertices of ${\rm GK}(G)$, then the $k$-tuple:  $${\rm D}(G)=\left(d_G(p_1),
d_G(p_2), \ldots, d_G(p_k )\right),$$
is called the {\em degree pattern of} $G$, where $d_G(p_i)$ denotes the degree
of $p_i$ in ${\rm GK}(G)$.
 We denote by $\mathfrak{OD}(G)$
the set of pairwise non-isomorphic finite
groups with the same order and degree pattern as $G$, and put
$h(G)=|\mathfrak{OD}(G)|$. Since there are only
finitely many isomorphism types of groups of order $|G|$,  $1\leqslant h(G)<\infty$.
This leads to the following definition.
\begin{definition}{\rm
A group $G$ is called {\em $k$-fold OD-characterizable}, if $h(G)=k$. }\end{definition}

A $1$-fold OD-characterizable group is often simply called {\em OD-characterizable}, and it is called {\em quasi
OD-characterizable} if it is $k$-fold OD-characterizable for some
$k>1$. We say that for a group $G$ {\em the OD-characterization problem is solved} if we know the value of $h(G)$.

In recent years, a special attention has been paid toward the problem of  OD-characterization of simple and almost simple groups.  Table 8 at the end of this article lists simple groups which are currently known to be OD-characterizable or quasi
OD-characterizable. In particular, a list of finite simple unitary groups, for which the OD-characterization problem is solved, is given in Table 8. Based on the results summarized in that table, we see that $h(U_3(3))=1$ (\cite{SZ-2008}),
$h(U_3(4))=1$ (\cite{ZS-2009}), $h(U_3(5))=1$ (\cite{ZS-2010}),
$h(U_3(8))=1$ (\cite{ZS-2009}),
$h(U_3(17))=1$ (\cite{atmost17}),  $h(U_3(q))=1$ where $q>5$ is a prime power with  $|\pi((q^2-q+1)/(3,q+1))|=1$ (\cite{degree}),  $h(U_4(2))=2$ (\cite{moh, regular}), $h(U_4(3))=1$ (\cite{SZ-2008}),  $h(U_4(4))=1$ (\cite{atmost17}),  $h(U_4(5))=1$ (\cite{BAk(sub)}),  $h(U_4(7))=1$ (\cite{MAk.Rah}),  $h(U_4(8))=1$ (\cite{moh}), $h(U_4(17))=1$ (\cite{moh}).

In this article, we study the OD-characterization problem for the simple  unitary groups $U_3(q)$ and $U_4(q)$.
In particular, we prove that the simple groups $U_3(q)$ for $q=31$, $37$, $43$, $47$, $49$, $59$, $61$, $64$, $73$, $89$ and $97$, and  $U_4(q)$ for  $q$ a prime power, $9\leqslant q\leqslant  97$, are OD-characterizable. Combined with the above known results, this indicates that the following theorem is valid.\\[0.3cm]
{\bf Theorem.} {\em  The simple unitary groups $U_3(q)$ and $U_4(q)$, with $q<10^2$, are OD-characterizable.}\\[0.3cm]
\indent A few words about the contents. In this article the word graph will mean a finite simple undirected graph. The sets of vertices of a graph $\Gamma$ will be denoted by $V(\Gamma)$, and $\deg_{\Gamma}(v)$ (where $v\in V(\Gamma)$) will denote the degree of the vertex $v$ in $\Gamma$. 
The maximum degree of  $\Gamma$, denoted by $\Delta (\Gamma)$, and the minimum degree of $ \Gamma$, denoted by $\delta (\Gamma)$, are the maximum and minimum degree of its vertices. Given a group $G$,  for the sake of convenience, we write $\delta(G)=\delta({\rm GK}(G))$ and $\Delta(G)=\Delta({\rm GK}(G))$.  For a prime $p$, we denote by $\mathfrak{S}(p)$ the set of nonabelian finite simple groups $G$ such that $p\in \pi (G)\subseteq \{2, 3, 5, \ldots, p\}$. The {\em $r$-part} of a natural number $n$ means the largest power of a prime $r$ dividing $n$. We will often write $n_{(r)}$ for the $r$-part of $n$.
Denote by $s(G)$ the number of connected
components of ${\rm GK}(G)$ and by $\pi_i=\pi_i(G)$, $i=1, 2, \ldots, s(G)$,
the $i$th connected component of ${\rm GK}(G)$. If $G$ is a group of even
order, then we put $2 \in \pi_1(G)$.  Denote by $\omega_i(G)$ a set consisting of $n\in \omega (G)$ such that every prime divisor of $n$ lies in $\pi_i(G)$. It is now easy to see that the order of a group $G$ can be expressed as a product
of some coprime natural numbers $m_i(G)$, $i=1, 2, \ldots,
s(G)$, with $\pi(m_i)=\pi_i(G)$, where $\pi(m_i)$ signifies the set of all prime divisors of $m_i$.
The numbers $m_1, \ldots,
m_{s(G)}$ are called the {\it order components of} $G$.

The sequel of this article is organized as follows.
In Section 2, we give several auxiliary
results to be used later.  In Section 3,
we recall some basic results on certain finite simple groups, especially, on their spectra. Section 4 is devoted to the arithmetical structure of
unitary simple groups $U_n(q)$. Finally, the conclusions are discussed in Sections 5 and 6.


\section{\sc Preliminaries}
The concept of the prime graph was first introduced by Gruenberg and Kegel.  During their study on prime graph, they showed that a finite group with disconnected prime graph is either a Frobenius group or a $2$-Frobenius group, or has a unique nonabelian composition factor with disconnected prime graph. This result is published by Williams in \cite {wili}.
\begin{pro}[Gruenberg and Kegel]\label{prop2} {\rm (\cite {wili}, Theorem A)}
If $G$ is a finite group with disconnected graph ${\rm GK}(G)$, then one of the following occurs:
\begin{itemize}
\item[{\rm (1)}] $s(G) = 2$, $G$ is a Frobenius group.
\item[{\rm (2)}] $s(G) = 2$, $G = ABC$, where $A$ and $AB$ are normal subgroups of $G$, $B$ is a normal subgroup of $BC$,
and $AB$ and $BC$ are Frobenius groups.
\item[{\rm (3)}] $s(G)\geqslant 2$, there exists a simple group $P$ such that $P\leqslant G/K \leqslant {\rm  Aut}(P)$ for some nilpotent normal $\pi_1(G)$-subgroup $K$ of $G$, and $G/P$ is a $\pi_1(G)$-group. Moreover, ${\rm GK}(P)$ is disconnected, $s(P)\geqslant s(G)$, and for every $2\leqslant i \leqslant s(G)$, there exists $2\leqslant j\leqslant s(P)$, such that $\omega _i(G)=\omega_j(P)$.
\end{itemize}
\end{pro}

A {\em clique} of a graph is a set of mutually adjacent vertices, that is, its induced subgraph is complete.
In \cite{suz}, Suzuki studied the structure of the prime graph of a simple group, and showed that all connected components
of a disconnected prime graph are cliques, except the first connected component.
\begin{pro}[Suzuki]\label{prop1} {\rm (\cite {suz}, Theorem B)}  Let $P$ be a finite simple group with $s=s(P)>1$.
Then, the connected components $\pi_2(P), \pi_3(P), \ldots, \pi_s(P)$ are  cliques.
\end{pro}

Now, suppose $G$ is an arbitrary group. Then, Propositions \ref{prop2} and \ref{prop1} show that the prime
graph of a group $G$ has the following structure:
\begin{equation}\label{e1} {\rm GK}(G)={\rm GK}[\pi_1]\oplus K_{n_2}\oplus \cdots\oplus K_{n_s},\end{equation}
where ${\rm GK}[\pi_1]$ denotes the induced subgraph ${\rm
GK}(G)[\pi_1(G)]$, $n_i=|\pi_i(G)|$ and $s=s(G)$.  It follows from Eq. (\ref{e1}) that:
\begin{itemize}
\item  if $p\in \pi_{i}(G)$, $i=2, \ldots, s$, then $d_G(p)=n_i-1$, and
\item if $p\in \pi_{1}(G)$, then $d_G(p)\leqslant |\pi_1(G)|-1$.
\end{itemize}
Hence, the degree sequence of the prime graph ${\rm GK}(G)$ of $G$ contains $n_i$ times the degree $n_i-1$, $2\leqslant i\leqslant s$, and $|\pi_1(G)|$ vertices of degree at most $|\pi_1(G)|-1$.  Following \cite{four}, for an integer $m\geqslant 0$, we define:  $$D_m(G):=\{p\in \pi(G) \ | \ d_G(p)=m\}.$$ Since ${\rm GK}(G)$ is a simple graph,  $0\leqslant m\leqslant |\pi(G)|-1$.
Some information on the degree pattern of $G$ can be obtained immediately through the number of elements in $D_m(G)$ for some $m$.  For instance, we have:
\begin{itemize}
\item[{\rm (1)}] $m=0$. Clearly, $p\in D_0(G)$ if and only if $\{p\}$ is a connected component of ${\rm GK}(G)$, so $|D_0(G)|\leqslant s(G) \leqslant 6$ (see \cite{wili}).
\item[{\rm (2)}] $m=|\pi(G)|-1$. If $D_m(G)\neq \emptyset$,  then ${\rm GK}(G)$ is connected.
\item[{\rm (3)}] $m=|\pi_i(G)|-1$ for $i\geqslant 2$.  It follows from Eq. (\ref{e1}) that  $|D_{m}(G)|\geqslant m+1$.
\end{itemize}

Finally, an immediate consequence of Eq. (\ref{e1}) is the following.

\begin{corollary}\label{coro1}
If $|D_m(G)|\leqslant m$ for some $m$, then $D_m(G)\subseteq \pi_1(G)$.  In particular, if
for every $m$,  $|D_m(G)|\leqslant m$, then ${\rm GK}(G)$ is connected.
\end{corollary}

Generally, in a graph a set of vertices is {\em independent} if no two vertices in the set are adjacent. The {\em independence number} $\alpha(\Gamma)$ of a graph $\Gamma$ is the maximum cardinality of an independent set of vertices in $\Gamma$. Given a group $G$,  we put $t(G)=\alpha({\rm GK}(G))$ and denote  by $t(r, G)$ the maximal number of prime divisors of $G$ containing $r$ that are pairwise nonadjacent in  ${\rm GK}(G)$. 
\begin{lm}\label{vasi}{\rm (\cite{vasi})} Let $G$ be a group with
$t(G)\geqslant 3$ and $t(2, G)\geqslant 2$, and let $K$ be the
maximal normal solvable subgroup of $G$. Then, there exists a
simple group $P$ such that $P\leqslant G/K\leqslant {\rm Aut} (P)$.
\end{lm}

\begin{lm}\label{31-97}
Let $G$ be a finite simple group of Lie type. Suppose that  $r\in \pi(G)\subseteq \pi(U_4(q))$, where $q$ is a prime power, and  $r$ and $q$ satisfy one of the following conditions. Then $G$ is isomorphic to one of the following simple groups in each case:
\begin{itemize}
\item[{\rm (1)}] $q=49$, $r=1201:$    $L_2(7^4)$, $B_2(7^2)$, $U_4(7^2)$.
\item[{\rm (2)}]  $q=59$, $r=1741:$    $L_2(59^2)$,  $B_2(59)$, $U_4(59)$.
\item[{\rm (3)}]  $q=61$, $r=1861:$  $L_2(61^2)$,  $B_2(61)$, $U_4(61)$.
\item[{\rm (4)}] $q=67$, $r=4423:$ $U_3(67)$,  $U_4(67)$.
\item[{\rm (5)}] $q=71$, $r=2521:$ $L_2(71^2)$,  $B_2(71)$, $U_4(71)$.
\item[{\rm (6)}] $q=79$, $r=6163:$   $U_3(79)$, $U_4(79)$.
\item[{\rm (7)}]  $q=81$, $r=6481:$  $U_3(3^4)$, $U_4(3^4)$.
\item[{\rm (8)}] $q=83$, $r=2269:$  $U_4(83)$.
\end{itemize}
\end{lm}
\begin{proof} Since the proofs of $(1)$--$(8)$ are similar, only
the proof for $(1)$ is presented.
Suppose that $G=L(q')$ is a finite simple group of Lie type over the finite field of order $q'=p^n$, where $p$ is a prime and $n$ is a natural number, such that  $1201\in \pi(G)\subseteq \pi(U_4(49))=\{2, 3, 5, 7, 13, 181, 1201\}$. Since $p\in \pi(G)\subseteq \{2, 3, 5, 7, 13, 181, 1201\}$,  we consider three cases separately. 

\begin{itemize}
\item[{\rm (1)}]  $p\in \{2, 3, 5, 13, 181\}$. If $p=2$, then one can easily check that  the order of $2$ modulo $1201$ is $300$, 
and the least integer $k$ for which $2^k+1\equiv 0\pmod{1201}$ is
150. Thus,  if $1201|2^k-1$ (for which $2^k-1$ divides $|G|$), then $k$ must be a multiple of $300$. Thus, in view of the list of the groups of Lie type together with their orders (see \cite{atlas}), no candidates for $G$ will arise.  Similarly, for $p\in \{3, 5, 13, 181\}$ we do not get a group. 
\item[{\rm (2)}]  $p=7$. In this case, the order of $7$ modulo $1201$ is $8$, and the least integer $k$ for which $7^k+1\equiv 0\pmod{1201}$ is
4,  and similar consideration will give the groups $G = L_2(7^4), B_2(7^2)$ and $U_4(7^2)$. 

\item[{\rm (3)}]  $p=1201$. In this case $q'$ must be a power of $1201$ and again we do not get a group. 
\end{itemize}
The lemma is proved.
\end{proof} 

Let $\Gamma$ be a simple graph with vertex set $V(\Gamma)=\{v_1, v_2, \ldots, v_n\}$ in any order. Put $d_i=\deg_{\Gamma} (v_i)$, $1\leqslant i \leqslant n$. The  sequence $D(\Gamma)=(d_1, d_2, \ldots, d_n)$  is called the {\em degree sequence} of $\Gamma$.
A nonnegative sequence of integers $(x_1, x_2, \ldots, x_n)$  is said to be {\em graphic} if there exists a graph $\Gamma$ with  $D(\Gamma)=(x_1, x_2, \ldots, x_n)$, and then the graph $\Gamma$ is called a {\em realization}. Two graphs with the same degree sequence are said to be {\em degree equivalent}. For example, the prime graphs
${\rm GK}(U_4(71))$ and ${\rm GK}(U_4(79))$ are degree equivalent (see Table 6). The following result is immediate.
\begin{lm}\label{coro2}  Let $D=(d_1, d_2, \ldots, d_n)$ be a graphic sequence.
Assume that $\Gamma$ is a realization of $D$ with a vertex set $V(\Gamma)=\{v_1, v_2, ..., v_n\}$ such that $\deg (v_i)=d_i$ for $1\leqslant  i\leqslant n$.  If the sequence $$D_{i,j}=(d_1, d_2, \ldots, d_{i-1}, d_i-1, d_{i+1}, \ldots, d_{j-1}, d_j-1, d_{j+1}, \ldots, d_n), \ \ (1\leqslant i<j\leqslant n),$$
is not graphic, then $v_i$ and $v_j$ are nonadjacent in $\Gamma$.
\end{lm}

\begin{lm}\label{connect}
Let $\Gamma$ be a graph. If $\Delta (\Gamma)+\delta(\Gamma)\geqslant |V(\Gamma)|-1$, then $\Gamma$ is connected.
\end{lm}
\begin{proof}  Assume the contrary and let $C_1, C_2, \ldots, C_s$ be the connected components of
$\Gamma$, where $s\geqslant 2$. Let $p\in C_i$ be a vertex of degree $\Delta (\Gamma)$. This forces $|C_i|\geqslant |\Delta(\Gamma)|+1$.
Since $\Gamma$ is disconnected, we may choose $q\in C_j$ for $j\neq i$.
Then, we have  $$\delta(\Gamma)\leqslant \deg (q)< |C_j|\leqslant |V(\Gamma)|-|C_i|\leqslant \Delta (\Gamma)+\delta(\Gamma)+1-\Delta (\Gamma)-1=\delta(\Gamma),$$  and this is a contradiction.
\end{proof}
\begin{lm}\label{lemma3.6}
Let $u, v$ be two nonadjacent vertices of a graph $\Gamma$.
If  $\deg_{\Gamma}(u)+ \deg_{\Gamma}(v)\leqslant  |V(\Gamma)|-3$,
then $\alpha(\Gamma)\geqslant 3$.
\end{lm}
\begin{proof} This follows immediately from the pigeonhole principle.
\end{proof}


\section{\sc The Structure of $U_n(q)$}
We collect here some information about the unitary groups over a finite field.
The general unitary group ${\rm GU}_n(q)$ is defined as
$${\rm GU}_n(q)=\{A\in {\rm GL}_n(q^2) \  |  \ A\bar{A}^T=I_n\},$$
which is a subgroup of ${\rm GL}_n(q^2)$.  Here $\bar{A}$ denotes the matrix obtained from $A$ by raising each entry to the $q$th power.
The special unitary group ${\rm SU}_n(q)$ is the subgroup
of all matrices in ${\rm GU}_n(q)$ with determinant $1$.
 The orders of ${\rm GU}_n(q)$ and ${\rm SU}_n(q)$ are, respectively,
$$|{\rm GU}_n(q)|=q^{n\choose 2}\prod_{i=1}^{n}(q^i-(-1)^i)  \ \ \ {\rm and} \ \ \ |{\rm SU}_n(q)|=q^{n\choose 2}\prod_{i=2}^{n}(q^i-(-1)^i).$$
The center $Z$ of ${\rm SU}_n(q)$ is the set of scalar matrices $\lambda I$ for which both $\lambda^n=1$ and $\lambda^{q+1}=1$.
There are exactly $d=(n, q+1)$ such $\lambda$ in  $\Bbb{F}_{q^2}$, and so the subgroup $Z$ has
 order $d$.  The quotient ${\rm PSU}_n(q)={\rm SU}_n(q)/Z$
is called the {\em projective special unitary group}, sometimes written by $U_n(q)$, and its order is
\begin{equation}\label{order} |U_n(q)|=\frac{|{\rm SU}_n(q)|}{d} =\frac{1}{d} q^{n\choose 2}\prod_{i=2}^{n}(q^i-(-1)^i).\end{equation}
The projective special unitary group $U_n(q)$ is usually a simple group,
however the exceptions are $U_2(2)$, $U_2(3)$ and $U_3(2)$.
It is a standard result that the groups $U_2(q)$ and $L_2(q)$ are isomorphic.

\begin{theorem}\label{but}  Let
$q$ be a power of a prime $p$. If $n\geqslant2$,  then the spectrum of
$U_n(q)$ is exactly the set of all divisors of the
following numbers:
\begin{itemize}
\item[{\rm (1)}] $\frac{q^n-(-1)^n}{d(q+1)}$, where $d=(n, q+1)$,
\item[{\rm (2)}] $\frac{[q^{n_1}-(-1)^{n_1}, \ q^{n_2}-(-1)^{n_2}]}{(n/(n_1, n_2), \ q+1)}$ for all $n_1, n_2>0$
such that $n_1+n_2=n$,
\item[{\rm (3)}] $[q^{n_1}-(-1)^{n_1}, \ q^{n_2}-(-1)^{n_2},
\ \ldots, \ q^{n_s}-(-1)^{n_s}]$ for each $s\geqslant 3$ and
all $n_1, n_2, \ldots, n_s>0$
such that $n_1+n_2+\cdots+n_s=n$,
\item[{\rm (4)}]  $p^k \cdot \frac{q^{n_1}-(-1)^{n_1}}{d}$ for all $k,
n_1>0$ such that $p^{k-1}+1+n_1=n$, and where $d=(n, q+1)$,
\item[{\rm (5)}]  $p^k\cdot [q^{n_1}-(-1)^{n_1}, \ q^{n_2}-(-1)^{n_2}, \ \ldots, \ q^{n_s}-(-1)^{n_s}]$ for each
$s\geqslant 2$ and all $k, n_1, n_2, \ldots, n_s>0$
such that $p^{k-1}+1+n_1+n_2+\cdots+n_s=n$,
\item[{\rm (6)}]   $p^k$ if $k>0$ and $p^{k-1}+1=n$.
\end{itemize}
\end{theorem}
\begin{proof} See Corollary 3 in \cite{buturlakin}.
\end{proof}

From now on, we will concentrate on the simple unitary groups $U_3(q)$ and $U_4(q)$.
We list now a few immediate consequences of
Theorem \ref{but}.
 \begin{corollary}\label{u23} {\rm (see also \cite{aleeva, scientia, degree, ZU(3q)})}
 Let $q$ be a power of a prime $p$.  Then
\begin{itemize}
\item[{\rm (1)}] when $q$ is odd,   $\mu(U_3(q))= \left\{
\begin{array}{ll}
\left\{q^2-q+1, \ q^2-1, \ p(q+1)\right\} & \mbox{if} \  q \not \equiv -1 \pmod{3},\\[0.3cm]
\left\{\frac{q^2-q+1}{3}, \ \frac{q^2-1}{3}, \ \frac{p(q+1)}{3}, \
q+1\right\} & \mbox{if} \ q \equiv  -1 \pmod{3}.
\end{array}
\right.$
\item[{\rm (2)}] when $q$ is even,   $\mu(U_3(q))= \left\{
\begin{array}{ll}
\left\{q^{2}-q+1, \ q^{2}-1, \ 2(q+1), 4\right\} & \mbox{if}  \ \ \ q \not \equiv -1 \pmod{3},\\[0.3cm]
\left\{\frac{q^{2}-q+1}{3}, \ \frac{q^{2}-1}{3}, \ \frac{2(q+1)}{3}, \
q+1, 4\right\} & \mbox{if} \ \ \  q \equiv -1 \pmod{3}.
\end{array}
\right.$
\end{itemize}
\end{corollary}

\begin{corollary}\label{u4-odd} {\rm (see also \cite{zav-L4})}
Let $q$ be a power of an odd prime $p$. Denote $d=(4,
q+1)$. Then $\mu(U_4(q))$ contains the following (and only the
following) numbers:
\begin{itemize}
\item[{\rm (1)}]  $(q-1)(q^2+1)/d$,
$(q^3+1)/d$, $p(q^2-1)/d$, $q^2-1$;
\item[{\rm (2)}] $p(q+1)$, if and only if $d=4$;
\item[{\rm (3)}] $9$, if and only if $p=3$.
\end{itemize}
\end{corollary}

\begin{corollary}\label{u4-even}
Suppose that $q=2^n$ with $n$ a natural number. If $n=1$, then $\mu(U_4(2))=\{5, 9, 12\}$, while if $n\geqslant 2$, then $\mu(U_4(q))=\left\{(q-1)(q^2+1), \
q^3+1,  \ 2(q^2-1), \ 4(q+1)\right\}$.
\end{corollary}

The prime graph ${\rm GK}(U_3(q))$ has two connected components (see \cite{kondra, wili}):
$\pi_1(U_3(q))=\pi(p(q^2-1))$ and $\pi_2(U_3(q))=\pi\left((q^2-q+1)/d\right)$, where $d=(3, q+1)$. Moreover, the prime graph ${\rm GK}(U_4(q))$ has one connected component
in all cases except the two particular cases $q=2, 3$.  In fact, we have $\pi_1(U_4(2))=\{2, 3\}$ and $\pi_2(U_4(2))=\{5\}$, while $\pi_1(U_4(3))=\{2, 3\}$,  $\pi_2(U_4(3))=\{5\}$ and $\pi_3(U_4(3))=\{7\}$.

If $n$ is a nonzero integer and $r$ is an odd prime with $(r, n) = 1$, then $e(r, n)$ denotes the multiplicative order of $n$ modulo $r$,  i.e., a minimal natural number $k$ with $n^k\equiv 1\pmod{r}$.
Given an odd integer $n$, we put $e(2, n)=1$ if $n\equiv 1\pmod{4}$, and $e(2, n)=2$ if $n\equiv 3\pmod{4}$. Fix an integer $n$ with $|n|>1$. A prime $r$ with $e(r, n)=i$ is called a {\em primitive prime divisor of} $n^i-1$.  We write $r_i(n)$  to denote some primitive prime divisor of $n^i-1$, if such a prime exists, and $R_i(n)$ to denote the set of all such divisors. Instead of $r_i(n)$ and $R_i(n)$ we simply write $r_i$ and $R_i$ if it does not lead to confusion.  Bang  \cite{Bang} and Zsigmondy \cite{Zsigmondy} proved that primitive prime divisors exist except for a few cases$^1$\footnote{$^1$In fact, Bang \cite{Bang} proved in 1886 that $n^i-1$ has a primitive prime divisor for all $n\geqslant 2$ and $i>2$ except for $n=2$ and $i=6$.
Then, Zsigmondy \cite{Zsigmondy} proved in 1892 that for coprime integers $a> b\geqslant 1$  and $i> 2$, there exists a prime $r$ dividing $a^i-b^i$ but not $a^k-b^k$ for $1\leqslant k<i$, except when $a=2$, $b=1$, and $i=6$.}\!\!\!\!.
\begin{theorem}\label{zsig}  {\rm  (Bang--Zsigmondy)}.
Let $n$ and $i$ be integers satisfying $|n|>1$ and $i\geqslant 1$. Then $R_i(n)\neq \emptyset$, except when
$(n, i)\in \{(2, 1), (2, 6), (-2,2), (-2,3), (3, 1), (-3,2)\}$.
\end{theorem}

Following \cite{za-primegraph}, we represent the prime graph ${\rm GK}(G)$ in a {\em compact form}. By a compact form we mean a graph whose vertices are labeled with sets $U_i$. A vertex labeled $U_i$
 represents the complete subgraph of  ${\rm GK}(G)$ on $U_i$. An edge joining $U_i$ and $U_j$ is a set of edges of ${\rm GK}(G)$ connecting each vertex in $U_i$ to each vertex in $U_j$. When $U_i$ is a singleton (e.g. $\{s\}$),  we will often write $s$ instead of $\{s\}$.
 Figures 1-3, for instance, depict the compact forms of the prime graphs of the simple unitary groups $U_3(q)$, with $q=p^n$. Here, $U_1=R_1(q)\setminus \{2, 3\}$, $U_2=R_2(q)\setminus \{2, 3\}$, and $U_3=R_6(q)$.

\setlength{\unitlength}{4.5mm}
\begin{picture}(0,0)(-2.25,8)
\put(0,0){\circle*{0.35}}%
\put(6,0){\circle*{0.35}}%
\put(3,2){\circle*{0.35}}%
\put(3,4){\circle*{0.35}}%
\put(3,6){\circle*{0.35}}%
\put(6,4){\circle*{0.35}}%

\put(12,0){\circle*{0.35}}%
\put(18,0){\circle*{0.35}}%
\put(15,2){\circle*{0.35}}%
\put(15,4){\circle*{0.35}}%
\put(15,6){\circle*{0.35}}%
\put(18,4){\circle*{0.35}}%

\put(24,0){\circle*{0.35}}%
\put(30,0){\circle*{0.35}}%
\put(27,2){\circle*{0.35}}%
\put(27,4){\circle*{0.35}}%
\put(27,6){\circle*{0.35}}%
\put(30,4){\circle*{0.35}}%

\put(0,0){\line(1,0){6}}
\put(0,0){\line(3,2){3}}
\put(0,0){\line(3,4){3}}
\put(0,0){\line(1,2){3}}
\put(3,6){\line(0,-1){2}}
\put(6,0){\line(-3,2){3}}
\put(6,0){\line(-3,4){3}}
\put(6,0){\line(-1,2){3}}

\put(12,0){\line(1,0){6}}
\put(12,0){\line(3,2){3}}
\put(12,0){\line(3,4){3}}
\put(12,0){\line(1,2){3}}
\put(18,0){\line(-3,2){3}}
\put(18,0){\line(-3,4){3}}
\put(18,0){\line(-1,2){3}}

\put(24,0){\line(1,0){6}}
\put(24,0){\line(3,2){3}}
\put(24,0){\line(3,4){3}}
\put(24,0){\line(1,2){3}}
\put(27,6){\line(0,-1){4}}
\put(30,0){\line(-3,2){3}}
\put(30,0){\line(-3,4){3}}
\put(30,0){\line(-1,2){3}}

\put(-0.2,-1){\footnotesize 2}%
\put(5.7,-1){\footnotesize $U_2$}%
\put(2.75,1){\footnotesize $p$}%
\put(2.75,3){\footnotesize $3$}%
\put(2.75,6.5){\footnotesize $U_1$}%
\put(5.6,4.5){\footnotesize $U_3$}%

\put(11.8,-1){\footnotesize 2}%
\put(17.7,-1){\footnotesize $U_2$}%
\put(14.75,1){\footnotesize $p$}%
\put(14.75,3){\footnotesize $3$}%
\put(14.75,6.5){\footnotesize $U_1$}%
\put(17.6,4.5){\footnotesize $U_3$}%

\put(23.8,-1){\footnotesize 2}%
\put(29.7,-1){\footnotesize $U_2$}%
\put(26.75,1){\footnotesize $p$}%
\put(27.25,4){\footnotesize $3$}%
\put(26.75,6.5){\footnotesize $U_1$}%
\put(29.6,4.5){\footnotesize $U_3$}%

\put(0.2,-2.3){\footnotesize  (a) \ $(q+1)_{(3)}=1$}
\put(12.2,-2.3){\footnotesize (b)  \ $(q+1)_{(3)}=3$}
\put(24.2,-2.3){\footnotesize (c) \ $(q+1)_{(3)}>3$}

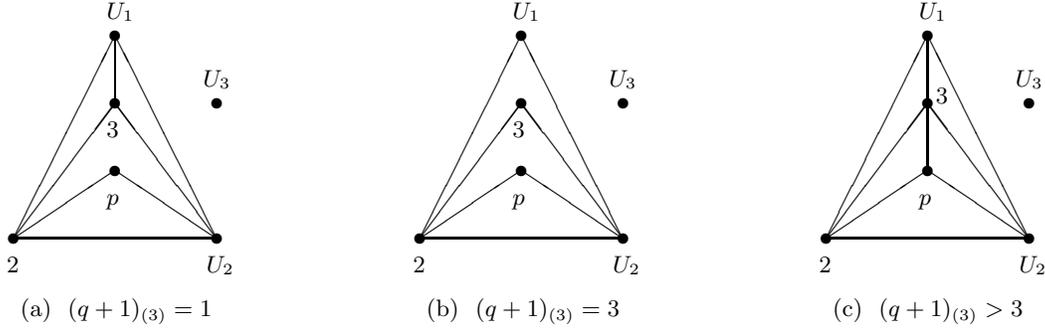
\put(1.5,-4.5){\footnotesize {\bf Fig. 1.} \  The diagram of a compact form for ${\rm GK}(U_3(q))$, where $q=p^n$ is odd and $ p\neq 3$.}
\end{picture}
\vspace{6cm}

\setlength{\unitlength}{4.5mm}
\begin{picture}(0,0)(-2.25,8)
\put(0,0){\circle*{0.35}}%
\put(6,0){\circle*{0.35}}%
\put(3,3){\circle*{0.35}}%
\put(3,6){\circle*{0.35}}%
\put(6,4){\circle*{0.35}}%

\put(12,0){\circle*{0.35}}%
\put(18,0){\circle*{0.35}}%
\put(15,3){\circle*{0.35}}%
\put(15,6){\circle*{0.35}}%
\put(18,4){\circle*{0.35}}%

\put(24,0){\circle*{0.35}}%
\put(30,0){\circle*{0.35}}%
\put(27,3){\circle*{0.35}}%
\put(27,6){\circle*{0.35}}%
\put(30,4){\circle*{0.35}}%

\put(0,0){\line(1,0){6}}
\put(3,6){\line(0,-1){3}}
\put(6,0){\line(-1,1){3}}
\put(6,0){\line(-1,2){3}}

\put(12,0){\line(1,0){6}}
\put(18,0){\line(-1,1){3}}
\put(18,0){\line(-1,2){3}}

\put(24,0){\line(1,0){6}}
\put(24,0){\line(1,1){3}}
\put(27,6){\line(0,-1){3}}
\put(30,0){\line(-1,1){3}}
\put(30,0){\line(-1,2){3}}

\put(-0.2,-1){\footnotesize 2}%
\put(5.7,-1){\footnotesize $U_2$}%
\put(2.75,2){\footnotesize $3$}%
\put(2.75,6.5){\footnotesize $U_1$}%
\put(5.6,4.5){\footnotesize $U_3$}%

\put(11.8,-1){\footnotesize 2}%
\put(17.7,-1){\footnotesize $U_2$}%
\put(14.85,2){\footnotesize $3$}%
\put(14.75,6.5){\footnotesize $U_1$}%
\put(17.6,4.5){\footnotesize $U_3$}%

\put(23.8,-1){\footnotesize 2}%
\put(29.7,-1){\footnotesize $U_2$}%
\put(26.75,2){\footnotesize $3$}%
\put(26.75,6.5){\footnotesize $U_1$}%
\put(29.6,4.5){\footnotesize $U_3$}%

\put(0.2,-2.3){\footnotesize  (a) \ $(q+1)_{(3)}=1$}
\put(12.2,-2.3){\footnotesize (b)  \ $(q+1)_{(3)}=3$}
\put(24.2,-2.3){\footnotesize (c) \ $(q+1)_{(3)}>3$}

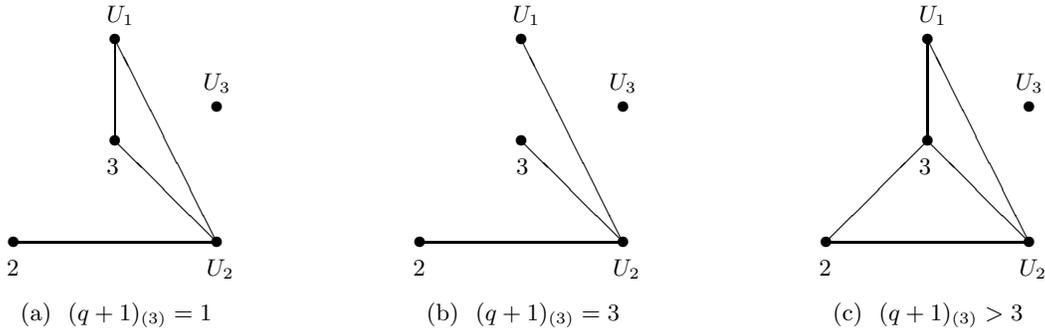
\put(3.5,-4.5){\footnotesize {\bf Fig. 2.} \  The diagram of a compact form for ${\rm GK}(U_3(q))$, where $q$ is even.}
\end{picture}
\vspace{5.5cm}

\setlength{\unitlength}{4.5mm}
\begin{picture}(0,0)(-2.25,8)
\put(12,0){\circle*{0.35}}%
\put(18,0){\circle*{0.35}}%
\put(15,3){\circle*{0.35}}%
\put(15,6){\circle*{0.35}}%
\put(18,4){\circle*{0.35}}%

\put(12,0){\line(1,0){6}}
\put(12,0){\line(1,1){3}}
\put(12,0){\line(1,2){3}}
\put(18,0){\line(-1,1){3}}
\put(18,0){\line(-1,2){3}}

\put(11.8,-1){\footnotesize 2}%
\put(17.7,-1){\footnotesize $U_2$}%
\put(14.75,2){\footnotesize $3$}%
\put(14.75,6.5){\footnotesize $U_1$}%
\put(17.6,4.5){\footnotesize $U_3$}%

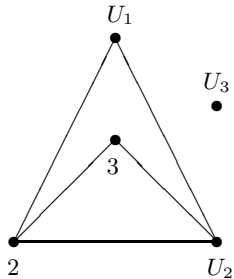
\put(6,-2.75){\footnotesize {\bf Fig.  3.} \  The diagram of a compact form for ${\rm GK}(U_3(3^n))$.}
\end{picture}
\vspace{5.5cm}

Similarly, we can draw the compact form of ${\rm GK}(U_4(q))$, where $q=p^n$ for some prime $p$, as
 illustrated in Figures 4-5.

\setlength{\unitlength}{4.5mm}
\begin{picture}(0,0)(-3,6)
\put(0,0){\circle*{0.35}}%
\put(6,0){\circle*{0.35}}%
\put(12,0){\circle*{0.35}}%
\put(3,3){\circle*{0.35}}%
\put(9,3){\circle*{0.35}}%

\put(18,0){\circle*{0.35}}%
\put(24,0){\circle*{0.35}}%
\put(30,0){\circle*{0.35}}%
\put(21,3){\circle*{0.35}}%
\put(27,3){\circle*{0.35}}%
\put(24,2){\circle*{0.35}}%

\put(0,0){\line(1,1){3}}
\put(3,3){\line(1,0){6}}
\put(3,3){\line(1,-1){3}}
\put(6,0){\line(1,1){3}}
\put(9,3){\line(1,-1){3}}

\put(18,0){\line(1,1){3}}
\put(21,3){\line(1,0){6}}
\put(21,3){\line(1,-1){3}}
\put(24,0){\line(1,1){3}}
\put(27,3){\line(1,-1){3}}

\put(24,2){\line(3,1){3}}
\put(24,2){\line(-3,1){3}}
\put(24,2){\line(0,-1){2}}

\put(-1,-1){\footnotesize $R_4(q)$}%
\put(2.2,3.75){\footnotesize $R_1(q)$}%
\put(5.8,-0.8){\footnotesize $p$}%
\put(8.2,3.75){\footnotesize $R_2(q)$}%
\put(11.3,-1){\footnotesize $R_6(q)$}%

\put(17,-1){\footnotesize $R_4(q)$}%
\put(20.2,3.75){\footnotesize $R_1(q)$}%
\put(23.8,-0.8){\footnotesize $p$}%
\put(26.2,3.75){\footnotesize $R_2(q)\setminus \{2\}$}%
\put(29.3,-1){\footnotesize $R_6(q)$}%
\put(23.8,2.4){\footnotesize $2$}%

\put(-2,-3){{\small \bf Fig. 4.} \mbox{The compact form of ${\rm GK}(U_4(q))$}}
\put(0.75,-4.2){ where $q=p^n$ and $(q+1)_{(2)}\neq 4$. }

\put(16,-3){{\small \bf Fig. 5.} \mbox{The compact form of ${\rm GK}(U_4(q))$}}
\put(18.75,-4.2){ where $q=p^n$ and $(q+1)_{(2)}=4$. }
\end{picture}
\vspace{5.5cm}

In Tables 1-3, we have determined the degrees of vertices of the prime graph ${\rm GK}(U_3(q))$ and ${\rm GK}(U_4(q))$. For convenience, we write $A_{i_1,i_2, \ldots, i_k}=A_{i_1}\cup A_{i_2}\cup \cdots \cup A_{i_k}$.

\begin{center}
{\bf Table 1.}  {\em The vertex degrees of
 ${\rm GK}(U_3(q))$, where $q=p^n$ is odd.}\\[0.2cm]
\begin{tabular}{|l|llllll|} \hline
Restrictions & $2$ & $3$ & $p>3$  &  $v\in U_1$  & $v\in U_2$  &  $v\in U_3$ \\[0.1cm]
\hline $(q+1)_{(3)}=1$  & $|U_{1,2}|+2$ & $|U_{1,2}|+1$ & $|U_2|+1$ & $|U_{1,2}|+1$ & $|U_{1,2}|+2$  & $|U_3|-1$ \\[0.1cm]
 $(q+1)_{(3)}=3$ &  $|U_{1,2}|+2$ & $|U_2|+1$ & $|U_2|+1$ & $|U_{1,2}|+1$ & $|U_{1,2}|+2$ & $|U_3|-1$\\[0.1cm]
  $(q+1)_{(3)}>3$ & $|U_{1,2}|+2$ & $|U_{1,2}|+2$  & $|U_2|+2$  & $|U_{1,2}|+1$ & $|U_{1,2}|+2$ & $|U_3|-1$ \\[0.1cm]

  $p=3$ &  $|U_{1,2}|+1$ & $|U_2|+1$ &  & $|U_{1,2}|$ & $|U_{1,2}|+1$ & $|U_3|-1$ \\[0.1cm]
\hline
\end{tabular}
\end{center}

\vspace{0.2cm}

\begin{center}
{\bf Table 2.}  {\em The vertex degrees of
 ${\rm GK}(U_3(q))$, where $q$ is even.}\\[0.2cm]
\begin{tabular}{|l|lllll|} \hline
Restrictions on $q$ & $2$ & $3$  &  $v\in U_1$  & $v\in U_2$  &  $v\in U_3$ \\[0.1cm]
\hline $(q+1)_{(3)}=1$  & $|U_2|$ & $|U_{1,2}|$ & $|U_{1,2}|$ & $|U_{1,2}|+1$  & $|U_3|-1$ \\[0.1cm]
 $(q+1)_{(3)}=3$ &  $|U_2|$ & $|U_2|$ & $|U_{1,2}|-1$ & $|U_{1,2}|+1$  & $|U_3|-1$\\[0.1cm]
  $(q+1)_{(3)}>3$ & $|U_2|+1$ & $|U_{1,2}|+1$  & $|U_{1,2}|$  & $|U_{1,2}|+1$ & $|U_3|-1$ \\
\hline
\end{tabular}
\end{center}

\vspace{0.2cm}

\begin{center}
{\bf Table 3.}  {\em The vertex degrees of
 ${\rm GK}(U_4(q))$, where $q=p^n$.}\\[0.2cm]
\begin{tabular}{|l|llllll|} \hline
Restrictions on $q$ & $2$ & $p$ & $v\in R_1\setminus \{2\}$  &  $v\in R_2\setminus \{2\}$  & $v\in R_6$  &  $v\in R_4$ \\[0.1cm]
\hline $(q+1)_{(2)}\neq 4, 2\nmid q$  & $|R_{1,2,4}|$ & $|R_{1,2}|$ & $|R_{1,2,4}|$ & $|R_{1,2,6}|$ & $|R_{2,6}|-1$  & $|R_{1,4}|-1$ \\[0.1cm]
 $(q+1)_{(2)}=4, 2\nmid q$ &  $|R_{1,2}|$ & $|R_{1,2}|$ & $|R_{1,2,4}|$ & $|R_{1,2,6}|$ & $|R_{2,6}|-1$ & $|R_{1,4}|-1$\\[0.1cm]
$(q+1)_{(2)}\neq 4, 2|q$ &  $|R_{1,2}|$ &  & $|R_{1,2,4}|$ & $|R_{1,2,6}|$ & $|R_{2,6}|-1$ & $|R_{1,4}|-1$\\
\hline
\end{tabular}
\end{center}


\section{\sc OD-Characterization of 3-Dimensional Unitary Groups}
As pointed out in the introduction,  in \cite{degree} it has been proved that  if $q>5$ is a prime power satisfying $|\pi((q^2-q+1)/(3,q+1))|=1$, then $h(U_3(q))=1$. For instance, when $q<10^2$ is a prime power we can find the following values
$q\in \{7, 8, 9, 11, 13, 16, 19, 23, 25, 29, 32, 41, 53, 67, 71, 79, 81, 83\}$,
for which $q$ satisfies the above condition, and so $h(U_3(q))=1$.
On the other hand,  in \cite{atmost17, SZ-2008, ZS-2009, ZS-2010}
the OD-characterization problem for simple groups $U_3(q)$ was solved for $q$ equals to $3, 4, 5, 17$. The main goal of this section is to argue for the validity of the following theorem.
\begin{theorem}\label{OD-U(3,q)}
The simple unitary groups $U_3(q)$ for $q\in \{31, 37, 43, 47, 49, 59, 61, 64, 73, 89, 97\}$ are OD-characterizable.
\end{theorem}

Before we start the proof of Theorem \ref{OD-U(3,q)}  we need some facts on
these 3-dimensional unitary groups. Fist of all, we note that  $$|U_3(q)|=\frac{1}{d}q^3(q^2-1)(q^3+1),$$ where
$d=(3, q+1)$. For the convenience of the reader, we have listed the orders and degree patterns of simple unitary groups $U_3(q)$ for  $q\in \{31, 37, 43, 47, 49, 59, 61, 64, 73, 89, 97\}$ in Table 4.

\begin{center}
{\bf Table 4.}  The order and degree pattern of some simple unitary groups.
$$\begin{array}{llll}
\hline
G & |G| & \mu(G) & D(G) \\[0.3cm]
\hline  & & & \\[-0.3cm]
U_3(31) &  2^{11} \cdot 3\cdot 5 \cdot 7^2 \cdot 19 \cdot 31^3  &\{7^2\cdot 19, 2^6\cdot 3\cdot 5, 2^5\cdot 31\} & (3, 2, 2, 1, 1, 1)  \\[0.2cm]
U_3(37) &  2^4 \cdot 3^2 \cdot 19^2 \cdot 31 \cdot 37^3 \cdot 43  & \{31\cdot 43, 2^3\cdot 3^2\cdot 19, 2\cdot 19\cdot 37\}  & (3, 2, 3, 1, 2, 1)  \\[0.2cm]
U_3(43) &  2^5 \cdot 3 \cdot 7  \cdot 11^2 \cdot 13  \cdot 43^3 \cdot 139  & \{13\cdot 139, 2^3\cdot 3\cdot 7\cdot 11, 2^2\cdot 11\cdot 43\} &(4, 3, 3, 4, 1, 2, 1)  \\[0.2cm]
U_3(47) &  2^9 \cdot 3^2 \cdot 7 \cdot 23 \cdot 47^3 \cdot 103 & \{7\cdot 103, 2^5\cdot 23, 2^4\cdot 47, 2^4\cdot 3\}&(3, 1, 1, 1, 1, 1)  \\[0.2cm]
U_3(49) &  2^6 \cdot 3 \cdot 5^4 \cdot 7^6 \cdot 13 \cdot 181  & \{13\cdot 181, 2^5\cdot 3\cdot 5^2,  2\cdot 5^2\cdot 7\}&(3, 2, 3, 2, 1, 1)  \\[0.2cm]
U_3(59) &  2^5 \cdot 3^2\cdot 5^2 \cdot 7 \cdot 29 \cdot 59^3 \cdot 163  & \{7\cdot 163, 2^3\cdot 5\cdot 29, 2^2\cdot 5\cdot 59, 2^2\cdot 3\cdot 5\} &(4, 2,  4, 1, 2, 2, 1)  \\[0.2cm]
U_3(61) &  2^4 \cdot 3 \cdot 5 \cdot 7 \cdot 31^2 \cdot 61^3 \cdot 523  & \{7\cdot 523, 2^3\cdot 3\cdot 5\cdot 31,  2\cdot 31\cdot 61\} &(4, 3, 3, 1, 4, 2, 1)  \\[0.2cm]
U_3(64) &  2^{18} \cdot 3^2 \cdot 5^2 \cdot 7 \cdot 13^2 \cdot 37 \cdot 109  & \{37\cdot 109, 3^2\cdot 5\cdot 7\cdot 13,  2\cdot 5\cdot 13\}&(2, 3, 4, 3, 4, 1, 1)  \\[0.2cm]
U_3(73) &  2^5 \cdot 3^2 \cdot 7 \cdot 37^2 \cdot 73^3 \cdot 751  & \{7\cdot 751, 2^4\cdot 3^2\cdot 37, 2\cdot 37\cdot 73\}&(3, 2, 1, 3, 2, 1)  \\[0.2cm]
U_3(89) &  2^5 \cdot 3^4 \cdot 5^2 \cdot 7 \cdot 11 \cdot 89^3 \cdot 373  & \{7\cdot 373, 2^4\cdot 3\cdot 5\cdot 11, 2\cdot 3\cdot 5\cdot 89, 2\cdot 3^2\cdot 5\} & (4, 4, 4, 1, 3, 3, 1)  \\[0.2cm]
U_3(97) &  2^7 \cdot 3 \cdot 7^4 \cdot 67 \cdot 97^3 \cdot 139  &  \{67\cdot 139, 2^6\cdot 3\cdot 7^2, 2\cdot 7^2\cdot 97\}& (3, 2, 3, 1, 2, 1)  \\[0.2cm]
\hline
\end{array}
$$
\end{center}

We will also use the following lemma in \cite{BAk2012}.

\begin{lm}\label{nonsolvable}
Let $G$ be a finite group with $n=|\pi(G)|$ and let $d_1\leqslant d_2\leqslant
\cdots\leqslant d_n$ be the degree sequence of ${\rm GK}(G)$. If
$d_1+d_{d_1+2}\leqslant n-3$, then
$t(G)\geqslant 3$.
\end{lm}

\noindent {\em Proof of Theorem \ref{OD-U(3,q)}}. Suppose $G$ is a finite group satisfying the following conditions:
$$ (1) \  |G|=|U_3(q)|=\frac{1}{d}q^3(q^2-1)(q^3+1)  \ \ \ \ \  \mbox{and}  \ \ \ \ \  (2) \ D(G)=D(U_3(q)),$$
where $q\in \{31, 37, 43, 47, 49, 59, 61, 64, 73, 89, 97\}$. Let  $d_1\leqslant d_2\leqslant \cdots \leqslant d_{|\pi(G)|}$ be the degrees of vertices of  ${\rm GK}(U_3(q))$. Using Table 5, in all cases, we see that $d_1=1$ and $d_3\leqslant |\pi(G)|-4$, and hence by Lemma \ref{nonsolvable}, $t(G)\geqslant 3$.  Moreover, in all cases, $t(2,G)\geqslant  2$ because $d_G(2)<|\pi(G)|-1$. Thus, by Lemma
\ref{vasi}, there exists a simple group $P$ such that $P\leqslant
G/K\leqslant {\rm Aut}(P)$, where $K$ is the maximal normal
solvable subgroup of $G$. Note that $|P|$ divides $|G|$, and so $P\in \mathfrak{S}(p)$, where $p=\max \pi(G)$. In the following, we will prove that $P\cong U_3(q)$, which implies that $K=1$ and since $|G| = |U_3(q)|$, $G$ is isomorphic to $U_3(q)$, as required.

\begin{center}
{\bf Table 5}. {\em   The values of $d_1$ and $d_{d_1+2}$ in ${\rm GK}(U_3(q))$ for some $q$}. \\[0.3cm]
$\begin{tabular}{ccllcc} \hline
&& & & \\[-0.26cm]
$S$ &  & $ d_1\leqslant \cdots\leqslant d_{|\pi(S)|}$ & $d_1$ & $d_{d_1+2}$ & $|\pi(S)|-3$ \\[0.1cm]
\hline\hline &&& & & \\[-0.36cm]
$U_3(31)$ &  & $(1, 1, 1, 2, 2, 3)$ & $1$   & $1$& $3$\\[0.1cm]
$U_3(37)$ &  & $(1, 1, 2, 2, 3, 3)$  & $1$& $2$ & $3$ \\[0.1cm]
$U_3(43)$ &  &$(1, 1, 2, 3, 3, 4, 4)$   & $1$ & 2 & $4$ \\[0.1cm]
$U_3(47)$ &   & $(1, 1, 1, 1, 1, 3)$ & $1$ & $1$& $3$\\[0.1cm]
$U_3(49)$  &  & $(1, 1, 2, 2, 3, 3)$ & $1$ & $2$ & $3$ \\[0.1cm]
$U_3(59)$ & & $(1, 1, 2, 2, 2, 4)$  & $1$& $2$ & $3$\\[0.1cm]
$U_3(61)$ &  & $(1, 1, 2, 3, 3, 4, 4)$  & $1$ & $2$& $4$\\[0.1cm]
$U_3(64)$ &  & $(1, 1, 2, 3, 3, 4, 4)$ & $1$ &  $2 $& $4$ \\[0.1cm]
$U_3(73)$ &  & $(1, 1, 2, 2, 3, 3)$ & $1$ & $2$ & $3$\\[0.1cm]
$U_3(89)$ &  & $(1, 1, 3, 3, 4, 4, 4)$ & $1$& $3$ & $4$ \\[0.1cm]
$U_3(97)$ & &  $(1, 1, 2, 2, 3, 3)$ & $1$ & $2$ & $3$\\[0.1cm]
\hline \end{tabular} $
\end{center}

We only consider two cases $q=31$ and $q=61$, and the other cases go similarly.

\noindent (a) {\em $q=31$}.
We distinguish two cases
separately.
\begin{itemize}
\item[{\rm (1)}] {\em Assume first that $7\sim 19$ in ${\rm GK}(G)$}.  In
this case we immediately imply that the prime graphs of
$G$ and $U_3(31)$ coincide (see Fig. 6).
Since ${\rm GK}(G)={\rm
GK}(U_3(31))$,  the hypothesis $|G|=|U_3(31)|$ implies that
${\rm OC}(G)={\rm OC}(U_3(31))$. Now, by the Main Theorem in
\cite{IKA}, $G$ is isomorphic to $U_3(31)$, as
required.

\item[{\rm (2)}] {\em Assume next that $7\nsim 19$ in ${\rm GK}(G)$}. In this case, we will show that $|P|$ is divisible by $7^2\cdot 19$.
Let  $\{r, s\}=\{7, 19\}$.  We first claim that $K$ is a $\{7, 19\}'$-group.
If $\{7, 19\}\subseteq \pi(K)$, then a Hall $\{7, 19\}$-subgroup of
$K$ would be an abelian group. Hence $7\sim 19$ in ${\rm GK}(K)$, and
so in ${\rm GK}(G)$, which is a contradiction. Suppose now that  $r\in\pi(K)$, $s\notin\pi(K)$ and $T\in
{\rm Syl}_{r}(K)$. By the Frattini argument,  $G=KN_G(T)$. This shows that
the normalizer $N_G(T)$ contains an element of order $s$, say $x$.
Then, $T\langle x\rangle$ is an abelian subgroup of $G$, which
leads to a contradiction as before. Thus  $K$ is a $\{7, 19\}'$-group, as claimed.
Obviously this forces that $|{\rm Aut}(P)|$ is divisible by $7^2\cdot 19$.
On the other hand, by \cite[Corollary 2.6]{kogani},
${\rm Out}(P)$ is also a $\{7, 19 \}'$-group, which implies that $|P|$ is divisible by $7^2\cdot 19$.
Considering the orders of simple groups in
$\mathfrak{S}(31)$, we conclude that $P$ is isomorphic
to $U_3(31)$, and so $K=1$ and since $|G|=|U_3(31)|$, $G$ is
isomorphic to $U_3(31)$. But then ${\rm GK}(G)={\rm GK}(U_3(31))$
is disconnected, which is impossible.
\end{itemize}

\noindent (b) {\em $q=61$}.
Again we consider  two cases
separately.
\begin{itemize}
\item[{\rm (1)}] {\em Assume first that $7\sim 523$ in ${\rm GK}(G)$}.  In
this case,  the group $G$  and the simple unitary group $U_3(61)$ have the same prime graph (see Fig. 7). Since  $|G|=|U_3(61)|$ we conclude that
${\rm OC}(G)={\rm OC}(U_3(61))$. Now, by the Main Theorem in
\cite{IKA}, $G$ is isomorphic to $U_3(61)$, as
required.

\item[{\rm (2)}] {\em Assume next that $7\nsim 523$ in ${\rm GK}(G)$}.
In the sequel, we will show that $|P|$ is divisible by $7\cdot 523$.
Let  $\{r, s\}=\{7, 523\}$.  We first claim that $K$ is a $\{7, 523\}'$-group.
If $\{7, 523\}\subseteq \pi(K)$, then a Hall $\{7, 523\}$-subgroup of
$K$ would be a cyclic group. Hence $7\sim 523$ in ${\rm GK}(K)$, and
so in ${\rm GK}(G)$, which is a contradiction. Suppose now that  $r\in\pi(K)$, $s\notin\pi(K)$ and $T\in
{\rm Syl}_{r}(K)$. By the Frattini argument,  $G=KN_G(T)$. This shows that
the normalizer $N_G(T)$ contains an element of order $s$, say $x$.
Then, $T\langle x\rangle$ is an abelian subgroup of $G$, which
leads to a contradiction as before. Thus  $K$ is a $\{7, 523\}'$-group, as claimed.
Obviously this forces that $|{\rm Aut}(P)|$ is divisible by $7\cdot 523$.
On the other hand, by \cite[Corollary 2.6]{kogani},
${\rm Out}(P)$ is also a $\{7, 523\}'$-group, which implies that $|P|$ is divisible by $7\cdot 523$.
Considering the orders of simple groups in
$\mathfrak{S}(61)$, we conclude that $P$ is isomorphic
to $U_3(61)$, and so $K=1$ and since $|G|=|U_3(61)|$, $G$ is
isomorphic to $U_3(61)$. But then ${\rm GK}(G)={\rm GK}(U_3(61))$
is disconnected, which is impossible. $\Box$
\end{itemize}

\setlength{\unitlength}{4.5mm}
\begin{picture}(0,0)(-16,8)
\put(8,0){\circle*{0.35}}%
\put(14,0){\circle*{0.35}}%
\put(11,2){\circle*{0.35}}%
\put(11,4){\circle*{0.35}}%
\put(11,6){\circle*{0.35}}%
\put(14,4){\circle*{0.35}}%
\put(8,0){\line(1,0){6}}
\put(8,0){\line(3,2){3}}
\put(8,0){\line(3,4){3}}
\put(8,0){\line(1,2){3}}
\put(11,6){\line(0,-1){2}}
\put(14,0){\line(-3,2){3}}
\put(14,0){\line(-3,4){3}}
\put(14,0){\line(-1,2){3}}
\put(7.8,-0.85){\footnotesize 2}%
\put(13.7,-0.85){\footnotesize $31$}%
\put(10.75,1){\footnotesize $61$}%
\put(10.75,3){\footnotesize $3$}%
\put(10.75,6.5){\footnotesize $5$}%
\put(14.5,3.8){\footnotesize $\{7, 523\}$}%
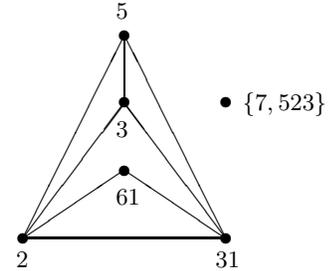
\put(5,-2.5){\footnotesize {\bf Fig. 7.} \ The prime graph of $U_3(61)$.}
\put(-10,0){\circle*{0.35}}%
\put(-7,2){\circle*{0.35}}%
\put(-7,4){\circle*{0.35}}%
\put(-7,6){\circle*{0.35}}%
\put(-4,4){\circle*{0.35}}%
\put(-10,0){\line(3,2){3}}
\put(-10,0){\line(3,4){3}}
\put(-10,0){\line(1,2){3}}
\put(-7,6){\line(0,-1){2}}
\put(-10.2,-0.85){\footnotesize 2}%
\put(-7.25,1){\footnotesize $31$}%
\put(-7.25,3){\footnotesize $3$}%
\put(-7.25,6.5){\footnotesize $5$}%
\put(-3.5,3.8){\footnotesize $\{7, 19\}$}%

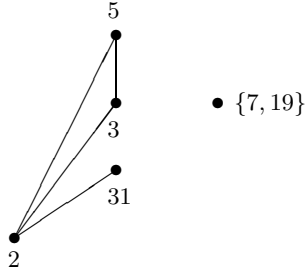
\put(-12,-2.5){\footnotesize {\bf Fig. 6.} \ The prime graph of $U_3(31)$.}
\end{picture}
\vspace{5cm}


\section{\sc OD-Characterization of 4-Dimensional Unitary Groups}
As mentioned in the introduction, it was already known that
$h(U_4(q))=1$ for $q=3, 4, 5, 7, 8$, and $17$ (see \cite{BAk(sub), MAk.Rah, moh, atmost17, SZ-2008}),
while $h(U_4(2))=2$ (\cite{moh, regular}).  For the  simple unitary groups $U_4(q)$, with $9\leqslant q\leqslant  97$,
we have determined their orders and degree patterns in Table 6.

The principal aim of this  section is to argue for the validity of the following theorem.
\begin{theorem}\label{OD-U(4,q)}
The simple unitary group $U_4(q)$ for a prime power $9\leqslant q\leqslant  97$ is OD-characterizable.
\end{theorem}

\begin{proof}  Suppose that
$G$ is a finite group such that:
$$ (1) \  |G|=|U_4(q)|=\frac{1}{d}q^6(q^2-1)(q^3+1)(q^4-1)  \ \ \ \ \  \mbox{and}  \ \ \ \ \  (2) \ D(G)=D(U_4(q)),$$
where $q$ is a prime power with  $9\leqslant q\leqslant 97$ and $d=(4, q+1)$. Using Table 6, in all cases, we see that $\Delta (G)+\delta(G)\geqslant |\pi(G)|-1$, and hence by Lemma \ref{connect}, ${\rm GK}(G)$ is connected.  Moreover, in all cases, $t(2,G)\geqslant  2$ because $d_G(2)<|\pi(G)|-1$. We now claim that $t(G)\geqslant 3$.  To prove this, we consider the following cases:

\begin{center}
{\bf Table 6.}  The order and degree pattern of some simple unitary groups.
$$\begin{array}{lll}
\hline
\mbox{Group} & \mbox{Order}  & \mbox{Degree Pattern} \\[0.23cm]
\hline  & &  \\[-0.3cm]
U_4(9) &  2^{9} \cdot 3^{12}\cdot 5^3 \cdot 41 \cdot 73    & (3, 2, 3, 1, 1)  \\[0.23cm]
U_4(11) &  2^7 \cdot 3^4 \cdot 5^2 \cdot 11^6 \cdot 37 \cdot 61   & (3, 4, 4, 3, 1, 1)  \\[0.23cm]
U_4(13) &  2^7 \cdot 3^2 \cdot 5  \cdot 7^3 \cdot 13^6  \cdot 17 \cdot 157   & (5, 5, 3, 4, 2, 3, 1)  \\[0.23cm]
U_4(16) &  2^{24} \cdot 3^2 \cdot 5^2 \cdot 17^3 \cdot 241 \cdot 257  & (3, 4, 4, 4, 1, 2)  \\[0.23cm]
U_4(17) &  2^{11} \cdot 3^7 \cdot 5 \cdot 7 \cdot 13 \cdot 17^6\cdot 29  &  (4, 4, 2, 2, 2, 2, 2)  \\[0.23cm]
U_4(19) &  2^7 \cdot 3^4 \cdot 5^3 \cdot 7^3 \cdot 19^6 \cdot 181   & (3, 4, 4, 1, 3, 1)  \\[0.23cm]
U_4(23) &  2^{10} \cdot 3^4\cdot 5 \cdot 11^2 \cdot 13^2 \cdot 23^6 \cdot 53   & (4, 4, 2, 5, 2, 3, 2)  \\[0.23cm]
U_4(25) &  2^9 \cdot 3^2 \cdot 5^{12} \cdot 13^3 \cdot 313 \cdot 601   & (4, 4, 3, 4, 2, 1)  \\[0.23cm]
U_4(27) &  2^7 \cdot 3^{18} \cdot 5 \cdot 7^3 \cdot 13^2 \cdot 19 \cdot 37 \cdot 73  & (3, 3, 2, 5, 5, 2, 2, 2)  \\[0.23cm]
U_4(29) &  2^7 \cdot 3^4 \cdot 5^3 \cdot 7^2 \cdot 29^6 \cdot 271 \cdot 421   & (5, 5, 5, 5, 4, 2, 2)  \\[0.23cm]
U_4(31) &  2^{16} \cdot 3^2 \cdot 5^2 \cdot 7^2 \cdot 13 \cdot 19 \cdot 31^6 \cdot 37   & (5, 5, 5, 2, 3, 2, 3, 3)  \\[0.23cm]
U_4(32) &  2^{30} \cdot 3^4 \cdot 5^2 \cdot 11^3 \cdot 31^2 \cdot 41\cdot 331   & (3, 4, 2, 4, 5, 2, 2)  \\[0.23cm]
U_4(37) &  2^7 \cdot 3^4 \cdot 5 \cdot 19^3 \cdot 31 \cdot 37^6 \cdot 43 \cdot 137  & (5, 5, 3, 5, 2, 3, 2, 3)  \\[0.23cm]
U_4(41) &  2^9 \cdot 3^4 \cdot 5^2 \cdot 7^3 \cdot 29^2 \cdot 41^6\cdot 547   & (5, 5, 5, 5, 2, 4, 2)  \\[0.23cm]
U_4(43) &  2^7 \cdot 3^2 \cdot 5^2 \cdot 7^2 \cdot 11^3 \cdot 13\cdot 37\cdot 43^6 \cdot 139   & (4, 6, 3, 6, 6, 2, 3, 4, 2)  \\[0.23cm]
U_4(47) &  2^{13} \cdot 3^4 \cdot 5 \cdot 7 \cdot 13 \cdot 17\cdot 23^2 \cdot 47^6 \cdot 103   & (5, 5, 3, 3, 3, 3, 6, 3, 3)  \\[0.23cm]
U_4(49) &  2^{11} \cdot 3^2 \cdot 5^6 \cdot 7^{12} \cdot 13 \cdot 181\cdot 1201    & (4, 4, 5, 3, 2, 2, 2)  \\[0.23cm]
U_4(53) &  2^7 \cdot 3^{10} \cdot 5 \cdot 13^2 \cdot 53^6 \cdot 281\cdot 919   & (5, 4, 3, 5, 3, 3, 1)  \\[0.23cm]
U_4(59) &  2^7 \cdot 3^4 \cdot 5^3 \cdot 7 \cdot 29^2 \cdot 59^6 \cdot 163 \cdot 1741   & (4, 6, 6, 3, 5, 4, 3, 1)  \\[0.23cm]
U_4(61) &  2^{7} \cdot 3^2 \cdot 5^2 \cdot 7 \cdot 31^3 \cdot 61^6\cdot 523 \cdot 1861   & (5, 5, 5, 2, 6, 4, 2, 1)  \\[0.23cm]
U_4(64) &  2^{36} \cdot 3^4 \cdot 5^3 \cdot 7^2 \cdot 13^3 \cdot 17\cdot 37 \cdot 109\cdot 241   & (4, 6, 6, 6, 6, 3, 3, 3, 3)  \\[0.23cm]
U_4(67) &  2^7 \cdot 3^2 \cdot 5 \cdot 11^2 \cdot 17^3 \cdot 67^6\cdot 449 \cdot 4423  & (4, 6, 3, 6, 5, 4, 3, 1)  \\[0.23cm]
U_4(71) &  2^{10} \cdot 3^7 \cdot 5^2 \cdot 7^2 \cdot 71^6 \cdot 1657\cdot 2521   & (5, 5, 5, 5, 4, 2, 2)  \\[0.23cm]
U_4(73) &  2^{9} \cdot 3^4 \cdot 5 \cdot 7\cdot  13 \cdot 37^3 \cdot 41\cdot 73^6\cdot 751   & (6, 6, 4, 2, 4, 5, 4, 3, 2)  \\[0.23cm]
U_4(79) &  2^{13} \cdot 3^2 \cdot 5^3 \cdot 13^2 \cdot 79^6 \cdot 3121\cdot 6163   & (5, 5, 5, 5, 4, 2, 2)  \\[0.23cm]
U_4(81) &  2^{11} \cdot 3^{24} \cdot 5^2 \cdot 17 \cdot 41^3 \cdot 193\cdot 6481   & (5, 3, 5, 3, 4, 3, 1)  \\[0.23cm]
U_4(83) &  2^7 \cdot 3^4 \cdot 5 \cdot 7^3 \cdot 13 \cdot 41^2\cdot 53 \cdot 83^6 \cdot 2269  & (4, 5, 3, 5, 3, 7, 3, 4, 2)  \\[0.23cm]
U_4(89) &  2^9 \cdot 3^7 \cdot 5^3 \cdot 7\cdot 11^2 \cdot 17 \cdot 89^6\cdot 233 \cdot 373   & (6, 6, 6, 3, 6, 3, 4, 3, 3)  \\[0.23cm]
U_4(97) &  2^{13} \cdot 3^2 \cdot 5 \cdot 7^6 \cdot 67 \cdot 97^6 \cdot 139 \cdot 941   & (5, 5, 3, 5, 2, 3, 2, 3)  \\[0.23cm]
\hline
\end{array}
$$
\end{center}

\begin{itemize}
\item[{\rm (a)}] $q\in \{13, 17, 27, 31, 37, 43, 47, 53, 59, 61, 67, 73, 81, 83, 97\}$. Let $d_1\leqslant d_2\leqslant \cdots\leqslant d_{|\pi(G)|}$ be the degree sequence of the prime graph ${\rm GK}(G)$. In all cases, Table 6 shows that
$$d_1+d_{d_1+2}\leqslant |\pi(G)|-3,$$ and the claim follows by applying Lemma \ref{nonsolvable}.

\item[{\rm (b)}] $q\in \{9, 11, 19, 23,  32, 49, 64, 89\}$.  Here, we have $2\delta\leqslant |\pi(G)|-3$ and $|D_\delta (G)|>\delta$.  We notice that the induced subgraph on  $D_\delta (G)$  is not complete.  Otherwise,  we obtain $|D_\delta (G)|=\delta+1$,
which shows that $D_\delta (G)\subset \pi(G)$ is a connected component of ${\rm GK}(G)$, a contradiction.
Hence, there are two nonadjacent vertices  $p$ and $q$ in $D_\delta(G)$,  such that
$$\deg(p)+\deg(q)=2\delta\leqslant |\pi(G)|-3.$$
Now the result follows from Lemma \ref{lemma3.6}.

\item[{\rm (c)}] $q\in \{16,  25, 29, 41, 71, 79\}$. In all cases, we have $|D_1(G)\cup D_2(G)|=2$.
Let $D_1(G)\cup D_2(G)=\{p, q\}$.  Using Lemma \ref{coro2} together with several computations,
we conclude that $p$ and $q$ are nonadjacent in ${\rm GK}(G)$.  Since $d_G(p)+d_G(q)\leqslant |\pi(G)|-3$,
it follows from Lemma \ref{lemma3.6} that $t(G)\geqslant 3$, as claimed.
\end{itemize}

Thus, by Lemma
\ref{vasi}, there exists a simple group $P$ such that $P\leqslant
G/K\leqslant {\rm Aut}(P)$, where $K$ is the maximal normal
solvable subgroup of $G$. Note that $|P|$ divides $|G|$, and so $P\in \mathfrak{S}(t)$, where
$t$ is the largest prime dividing $|G|$. In what follows, we will prove that $P\cong U_4(q)$, which implies that $K=1$ and since $|G| = |U_4(q)|$, $G$ is isomorphic to $U_4(q)$, as required.

We distinguish two cases
separately:
\begin{itemize}
\item[{\rm (1)}]  $q\in \{9, 11, 13, 16, 17, 19, 23, 25, 27, 29, 31, 32, 37, 41, 43, 47, 53, 64, 73,  89, 97\}$.
Analysis of different possibilities for $q$
proceeds along similar lines, so, we only handle one case.
Let $q=9$. In this case, we have
$|G|=2^9\cdot 3^{12}\cdot 5^3\cdot 41\cdot 73$ and $D(G)=(3, 2,
3, 1, 1)$.  There are only
two possibilities for the prime graph of $G$ as shown in Fig. 8. Here $\{r, s\}=\{2, 5\}$.

{\setlength{\unitlength}{7mm}
\begin{picture}(1,1)(11,-12)
\linethickness{0.3pt} %
\put(19.1,-14){\circle*{0.3}}
\put(21.9,-14){\circle*{0.3}}
\put(20.5,-12.6){\circle*{0.3}}
\put(24.7,-14){\circle*{0.3}}
\put(23.3,-12.6){\circle*{0.3}}

\put(20.5,-12.2){\small$r$}%
\put(23.2,-12.2){\small$s$}%
\put(18.7,-14.75){\small$41$}%
\put(21.8,-14.75){\small$3$}%
\put(24.5,-14.75){\small$73$}%

\put(19.1,-14){\line(1,1){1.4}}
\put(20.5,-12.6){\line(1,0){2.8}}
\put(20.5,-12.6){\line(1,-1){1.4}}
\put(21.9,-14){\line(1,1){1.4}}
\put(23.3,-12.6){\line(1,-1){1.4}}

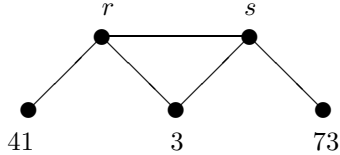
\put(16,-16){{\small \bf Fig. 8.} \mbox{All possibilities for
the prime graph of} $G$.}
\end{picture}}
\vspace{3.5cm}

 As before, one can show that $K$
is a $\{41, 73\}'$-group, and so $\{41, 73\}\subseteq \pi(P)$.
Moreover, since $t=73$ and $\pi(P)\subseteq \pi(G)$, we obtain $|P|=2^\alpha\cdot
3^\beta\cdot 5^\gamma\cdot 41\cdot 73$, where $2\leqslant
\alpha\leqslant 12$, $0\leqslant \beta\leqslant 2$ and
$0\leqslant \gamma\leqslant 3$. Now, by considering the orders of simple
groups $\mathfrak{S}(73)$ (see \cite{za}), we observe that the only possibility
for $P$ is $U_4(9)$.
\item[{\rm (2)}]  $q\in \{49, 59, 61, 67, 71, 79,  81, 83\}$.
Let $q=49$. In this case, we have
$|G|=2^{11} \cdot 3^2 \cdot 5^6 \cdot 7^{12} \cdot 13 \cdot 181\cdot 1201 $ and $D(G)= (4, 4, 5, 3, 2, 2, 2) $.
Let $\pi=\{r, s, t\}=\{13, 181, 1201\}$. We claim that  $K$ is a $\pi'$-subgroup of $G$.  If  $\pi\subseteq \pi(K)$,  then a Hall $\pi$-subgroup of $K$ is abelian, and so  $13\sim 181\sim 1201\sim 13$ in ${\rm GK}(G)$. This shows that $\pi$ is a connected component of ${\rm GK}(G)$, which is a contradiction.  We have a similar situation if $|\pi\cap \pi(K)|=2$.  In fact, if $r, s\in \pi(K)$,  then as before $r\sim s$ in ${\rm GK}(G)$. On the other hand, if $R\in {\rm Syl}_r(K)$ and $S\in {\rm Syl}_s(K)$, then
$G=KN_G(R)=KN_G(S)$ by the Frattini argument.  This shows that
the normalizers $N_G(R)$ and  $N_G(S)$  contain an element of order $t$, say $x$.
Then, the subgroups $R\langle x\rangle$  and  $S\langle x\rangle$ are abelian subgroups of $G$, which implies that
$t\sim s$ and $t\sim r$ in ${\rm GK}(G)$. Therefore $r\sim s\sim t\sim r$ in ${\rm GK}(G)$, which
leads to a contradiction as before. We finally assume $|\pi\cap \pi(K)|=1$.  Let $r\in \pi(K)$. In this case,  similar arguments show that  $t\sim r$ and $s\sim r$ in ${\rm GK}(G)$.  Moreover, it follows that $5\sim s$ and $5\sim t$ in ${\rm GK}(G)$, since $d_G(5)=5$. But then $2$ is adjacent to at most  3 vertices, which is impossible because $d_G(2)=4$. Therefore
$K$ is a $\pi'$-subgroup of $G$.

In other cases,  $K$ is a $\pi'$-group, for some set $\pi\subseteq \pi(G)$. In Table 7, we have determined the set $\pi$ for each case.
\begin{center}
{\bf Table 7.}  The set $\pi\subseteq \pi(G)$ for which $K$ is a $\pi'$-group.
$$\begin{array}{lll}
\hline
\mbox{Group} & \mbox{Order}  & \pi \\[0.3cm]
\hline  & &  \\[-0.3cm]
U_4(49) &  2^{11} \cdot 3^2 \cdot 5^6 \cdot 7^{12} \cdot 13 \cdot 181\cdot 1201    & \{13, 181, 1201\}  \\[0.3cm]
U_4(59) &  2^7 \cdot 3^4 \cdot 5^3 \cdot 7 \cdot 29^2 \cdot 59^6 \cdot 163 \cdot 1741   & \{7, 163, 1741\}    \\[0.3cm]
U_4(61) &  2^{7} \cdot 3^2 \cdot 5^2 \cdot 7 \cdot 31^3 \cdot 61^6\cdot 523 \cdot 1861   & \{7, 523, 1861\}    \\[0.3cm]
U_4(67) &  2^7 \cdot 3^2 \cdot 5 \cdot 11^2 \cdot 17^3 \cdot 67^6\cdot 449 \cdot 4423  & \{5, 449, 4423\}    \\[0.3cm]
U_4(71) &  2^{10} \cdot 3^7 \cdot 5^2 \cdot 7^2 \cdot 71^6 \cdot 1657\cdot 2521   & \{1657, 2521\}    \\[0.3cm]
U_4(79) &  2^{13} \cdot 3^2 \cdot 5^3 \cdot 13^2 \cdot 79^6 \cdot 3121\cdot 6163   & \{3121, 6163\}    \\[0.3cm]
U_4(81) &  2^{11} \cdot 3^{24} \cdot 5^2 \cdot 17 \cdot 41^3 \cdot 193\cdot 6481   & \{17, 193, 6481\}    \\[0.3cm]
U_4(83) &  2^7 \cdot 3^4 \cdot 5 \cdot 7^3 \cdot 13 \cdot 41^2\cdot 53 \cdot 83^6 \cdot 2269  &  \{5, 13, 53, 2269\}    \\[0.3cm]
\hline
\end{array}
$$
\end{center} Let $p=\max \pi(G)$.  Then,  Lemma 2.8 in \cite{kogani}, we conclude that $p\in  \pi(P)$. In the case when $q=49$, $1201\in \pi(P)$ and Lemma \ref{31-97} (1) shows that $P$ is isomorphic to one of the following groups:
$$L_2(7^4), \ B_2(7^2), \  U_4(7^2).$$
If $P$ is isomorphic to $L_2(7^4)$ or $B_2(7^2)$, then $181\notin \pi(P)$. This forces  $181\in \pi({\rm Out}(P))$, which is a contradiction.  Therefore, $P$ is isomorphic to $U_4(7^2)$.

Using Lemma \ref{31-97} and similar arguments as the previous case, we can verify that  $P$ is isomorphic to  $U_4(q)$, for $q=59$, $61$, $67$, $71$, $79$,  $81$ or $83$.  The result follows.
\end{itemize}
The proof is complete. \end{proof}

\begin{center}
\begin{tabular}{|l|l|c|l|}
\multicolumn{4}{c}{\textbf{Table 8.}  Simple groups for which the OD-characterization problem is solved.} \\[0.5cm]
\hline $G$ & Conditions on $G$ & $h(G)$ & References \\[0.01cm]
\hline
${\Bbb A}_n$ & $n=p, p+1, p+2$ ($p$ a prime, $p\geqslant 5$) & 1 & \cite{M.Z(2008), degree} \\[0.01cm]
& $5\leqslant n\leqslant 100, \ n\neq10$ & 1 & \cite{H.M(2010), kogani, MR-2011, M.Z(2009), SSWZ}\\[0.01cm]
& $n=106, \ 112, \ 116, \ 134$ & 1 & \cite{chen(2010), chen(2013)}\\[0.01cm]
& $n=10$ & 2 & \cite{connected} \\[0.01cm]
$L_2(q)$ & $q\neq 2, 3$& 1 & \cite{degree, ZS-2012}\\[0.01cm]
$L_3(q)$ & $\left|\pi\left(\frac{q^2+q+1}{d}\right)\right|=1$, $d=(3, q-1)$ & 1& \cite{degree} \\[0.01cm]
$L_4(q)$ & $q\leqslant 37$ & 1 & \cite{BAk2012, four, MAk.Rah} \\[0.01cm]
$L_n(2)$ & $n=p$ or $p+1$, $2^p-1$ is Mersenne prime & 1&\cite{MAk.Rah} \\[0.01cm]
$L_n(2)$ & $n=9, 10, 11$ & 1 & \cite{Khos, binary} \\[0.01cm]
$L_3(7^n)$ & $n=2, 3, 4$ & 1& \cite{QXL}\\[0.01cm]
$L_3(9)$ & & 1& \cite{SS-2009}\\[0.01cm]
$L_6(3)$ & & 1& \cite{BAk(sub)}\\[0.01cm]
$L_6(5)$ & & 1 & \cite{ZhangL65}\\[0.01cm]
$U_3(q)$ & $\left|\pi\left(\frac{q^2-q+1}{d}\right)\right|=1$, $d=(3,q+1)$, $q>5$ & 1& \cite{degree} \\[0.01cm]
$U_3(5)$ & & 1 & \cite{ZS-2010} \\[0.01cm]
$U_4(2)$  & & 2 & \cite{moh, regular} \\[0.01cm]
$U_4(q)$ & $q= 3, 4, 5 , 7, 8, 17$ & 1 &\cite{BAk(sub), MAk.Rah, moh, atmost17, SZ-2008}\\[0.01cm]
$U_6(2)$ &  & 1 & \cite{Zhang-Shi-2011}\\[0.01cm]
${^2G}_2(q)$ & $\left|\pi(q \pm \sqrt{3q}+1)\right|=1, q=3^{2m+1},
m\geqslant1$
& 1 & \cite{degree} \\[0.01cm]
${\rm Sz}(q)$ & $q=2^{2n+1}\geqslant 8$ & 1& \cite{degree} \\[0.01cm]
$S_4(q)$ & $|\pi((q^2+1)/2)|=1$, $q\neq 3$ & 1& \cite{2-fold}\\[0.01cm]
$S_6(3)$  & & 2 & \cite{degree} \\[0.01cm]
$S_6(4)$ & $S_6(4)\cong O_7(4)$  & 1& \cite{M-2010} \\[0.01cm]
$S_6(5)$  & & 2 & \cite{2-fold} \\[0.01cm]
$S_{2n}(q)$ & $n=2^m\geqslant 2, \ 2\mid q, \
|\pi(q^n+1)|=1$, $(n, q)\neq(2, 2)$ &1&\cite{2-fold} \\[0.01cm]
& $S_{2n}(q)\cong O_{2n+1}(q)$ &   &  \\[0.01cm]
$S_{2m}(q)$ &
$m=2^f\geqslant 2$, \
$\left|\pi\left(\frac{q^m+1}{2}\right)\right|=1$,
\ $q$ \mbox{odd} &  2 &  \cite{2-fold}\\ [0.3cm]
$S_{2p}(3)$ &
$\left|\pi\left(\frac{3^p-1}{2}\right)\right|=1$,
\ $p$ \mbox{odd prime} &  2 &  \cite{2-fold}\\ [0.1cm]
$O_7(3)$  & & 2 & \cite{degree} \\[0.01cm]
$O_7(5)$  & & 2 & \cite{2-fold} \\[0.01cm]
$G$ & $G$ is a sporadic group & 1& \cite{degree} \\[0.1cm]
$G$ & $|G|\leqslant 10^8 , G\neq {\Bbb A}_{10}, U_4(2)$ & 1& \cite{SZ-2008}\\[0.01cm]
$G$ & $|\pi(G)|=4, G\neq {\Bbb A}_{10}$ & 1 & \cite{ZS-2009}\\[0.01cm]
$G$ &  $\pi_1(G)=\{2\}$ & 1 & \cite{M.Z(2008)}\\[0.01cm]
$G$ & $\pi(G)\subseteq \pi(17!), G\neq {\Bbb
A}_{10}, U_4(2)$ & 1 & \cite{atmost17}\\[0.01cm]
$G$ & $\pi(G)\subseteq \pi(29!), G\neq {\Bbb
A}_{10}, \ U_4(2), \ S_6(3), \ O_7(3)$ & 1 & \cite{moh}\\[0.3cm]
$O_{2m+1}(q)$ &
$m=2^f\geqslant 2$, \
$\left|\pi\left(\frac{q^m+1}{2}\right)\right|=1$,
\ $q$ \mbox{odd} &  2 &  \cite{2-fold}\\ [0.3cm]
$O_{2p+1}(3)$ &
$\left|\pi\left(\frac{3^p-1}{2}\right)\right|=1$,
\ $p$ \mbox{odd prime} &  2 &  \cite{2-fold}\\ [0.3cm]
\hline \end{tabular}
\end{center}

\vspace{0.5cm}

\noindent {\sc Majid Akbari}\\[0.2cm]
{\sc Department of Mathematics, Payame Noor University, Tehran, Iran.}\\[0.3cm]
{\sc   Xiaoyou Chen}\\[0.2cm]
{\sc College of Science, Henan University of Technology, $450001$, Zhengzhou, China.}\\[0.2cm]
{ E-mail address:}  {\tt cxy19800222@163.com}\\[0.3cm]
{\sc Faisal Hassani}\\[0.2cm]
{\sc Department of Mathematics, Payame Noor University, Tehran, Iran.}\\[0.3cm]
{\sc A. R. Moghaddamfar}\\[0.2cm]
{\sc Faculty of Mathematics, K. N. Toosi
University of Technology,
 P. O. Box $16315$--$1618$, Tehran, Iran.}\\[0.1cm]
{ E-mail addresses:}  {\tt
moghadam@kntu.ac.ir}, and {\tt moghadam@ipm.ir}\\[0.3cm]

\begin{thebibliography}{99}
\bibitem{BAk2012} B. Akbari and A. R. Moghaddamfar, Recognizing by
order and degree pattern of some projective special linear
groups,  {\em Internat. J. Algebra Comput.}, 22(6)(2012), 22
pages.

\bibitem{BAk(sub)} B. Akbari and A. R. Moghaddamfar,
On recognition by order and degree pattern of finite simple
groups, {\em Southeast Asian Bull. Math.}, 39(2)(2015), 163--172.

\bibitem{four} B. Akbari and A. R. Moghaddamfar,  OD-Characterization of certain four dimensional linear
groups with related results concerning degree patterns, {\em
Front. Math. China}, 10(1)(2015), 1--31.

\bibitem{2-fold} M. Akbari and  A. R. Moghaddamfar,
Simple groups which are $2$-fold OD-characterizable, {\em Bull.
Malays. Math. Sci. Soc.}, 35(1)(2012), 65--77.

\bibitem{MAk.Rah} M. Akbari, A. R. Moghaddamfar and S. Rahbariyan,
A characterization of some finite simple groups through their
orders and degree patterns, {\em Algebra Colloq.}, 19(3)(2012),
473--482.

\bibitem{aleeva}  M. R. Aleeva, On the composition factors of finite groups with a set of element orders as in the group $U_3(q)$,  {\em Siberian Math. J.},  43(2) (2002),  195--211.

\bibitem{Bang}  A. S. Bang, Taltheoretiske Unders{\rm $\phi$}gelser. {\em Tidsskrift Math.},   4 (5)(1886), 70--80 and 130--137.

\bibitem{buturlakin} A. A. Buturlakin, Spectra of finite linear and unitary groups, {\em Algebra Logic},  47(2) (2008), 91--99.

\bibitem{atlas} J. H. Conway, R. T. Curtis, S. P. Norton, R. A.
Parker and R. A. Wilson, {\em Atlas of Finite Groups}, Clarendon
Press, oxford, 1985.

\bibitem{scientia} M. R. Darafsheh, A. R.  Moghaddamfar and A. R.  Zokayi, A recognition of simple groups ${\rm PSL}(3,q)$ by their element orders, {\rm Acta Math. Sci.},  24 (1) (2004), 45--51.

\bibitem{H.M(2010)} A. A. Hoseini and A. R. Moghaddamfar, Recognizing alternating
groups $A\sb {p+3}$ for certain primes $p$ by their orders and
degree patterns, {\em Front. Math. China}, 5(3)(2010), 541--553.

\bibitem{IKA} A. Iranmanesh, B. Khosravi and S. H. Alavi,  A characterization of ${\rm PSU}_3(q)$ for $q>5$, {\em Southeast Asian Bull. Math.},  26(1) (2002), 33--44.

\bibitem{Khos} B. Khosravi, Some characterizations of $L\sb 9(2)$ related to its prime graph,
{\em Publ. Math. Debrecen}, 75(3-4) (2009), 375--385.

\bibitem{kogani} R. Kogani-Moghaddam and A. R. Moghaddamfar,
Groups with the same order and degree pattern, {\em Sci. China
Math.}, 55(4)(2012), 701--720.

\bibitem{kondra} A. S. Kondrat\'ev, Prime graph components of finite simple groups, {\em Math. Sb.},  180(6)
(1989), 787--797.


\bibitem{QXL} Q. X. Li, OD-characterization of some linear groups, {\em South Asian Journal of Mathematics}, 
7(1)(2017),  34--42.

\bibitem{moh}  A. Mohammadzadeh and  A.R. Moghaddamfar,
Several quantitative characterizations of some specific groups,
{\em Comment. Math. Univ. Carolin.},  58(1) (2017) 19--34.

\bibitem{M-2010} A. R. Moghaddamfar, Recognizability of finite groups by order and
degree pattern, {\em Proceedings of the International Conference
on Algebra}, (2010), 422--433.

\bibitem{moghadam} A. R. Moghaddamfar,  On alternating and symmetric groups which are quasi OD-characterizable,
submitted for publication, {\em J. Algebra Appl.},  16(4) (2017), 1750065, 14 pp.

\bibitem{MR-2011} A. R. Moghaddamfar and S. Rahbariyan, More on the
OD-Characterizability of a finite group, {\em Algebra Colloq.},
18 (2011), 663--674.

\bibitem{atmost17} A. R. Moghaddamfar and S. Rahbariyan, A quantitative
characterization of some finite simple groups through order and
degree pattern, {\em Note Mat.},  34(2) (2014), 91--105.

\bibitem{binary} A. R. Moghaddamfar and S. Rahbariyan, OD-Characterization of some projective special
linear groups over the binary field and their automorphism
groups, {\em Comm. Algebra}, 43(6)(2015), 2308--2334.

\bibitem{M.Z(2008)} A. R. Moghaddamfar and A. R. Zokayi, Recognizing finite group through order and degree
pattern, {\em Algebra Colloq.}, 15(3)(2008), 449--456.

\bibitem{M.Z(2009)} A. R. Moghaddamfar and A. R. Zokayi, OD-Characterization of alternating and symmetric
groups of degree 16 and 22, {\em Front. Math. China}, 4(2009),
669--680.

\bibitem{connected} A. R. Moghaddamfar and A. R. Zokayi,
OD-characterization of certain finite groups having connected
prime graphs, {\em Algebra Colloq.},  17(1)(2010), 121--130.

\bibitem{degree}  A. R. Moghaddamfar, A. R. Zokayi and M. R. Darafsheh,
A characterization of finite simple groups by the degrees of
vertices of their prime graphs, {\em Algebra Colloq.},
12(3)(2005), 431--442.

\bibitem{SSWZ} C. Shao, W. Shi, L. Wang and L. Zhang, OD-Characterization of ${\Bbb
A}_{16}$, {\em Journal of Suzhou University (Natural Science
Edition)}, 24 (2008), 7--10.

\bibitem{SS-2009} C. Shao, W. Shi, L. Wang and L. Zhang, OD-Characterization of the simple group $L_3(9)$,
{\em Journal of Guangxi University (Natural Science Edition)},
34(2009), 120--122.

\bibitem{SZ-2008} W. J. Shi  and L. Zhang, OD-Characterization of all simple groups whose orders are less than $10^8$,
{\em Front. Math. China}, 3(2008), 461--474.

\bibitem{suz} M. Suzuki. On the prime graph of a finite simple
group---an application of the method of
Feit-Thompson-Bender-Glauberman. {\it Groups and
combinatorics---in memory of Michio Suzuki, Adv. Stud. Pure
Math.} {\bf  32} ( Math. Soc. Japan, Tokyo, 2001), 41--207.

\bibitem{vasi} A. V. Vasil\'ev and I. B. Gorshkov,  On the recognition of finite simple groups with a connected prime graph, {\em Sib. Math. J.}, 50 (2009), 233--238.

\bibitem{VS} A. V. Vasil\'ev and Staroletov, Recognizability of the groups $G_2(q)$ by the spectrum, {\em  Algebra Logic},  52(1) (2013), 1--14.

\bibitem{wili} J. S. Williams, Prime graph components of finite
groups, {\em J. Algebra}, 69(2)(1981), 487--513.

\bibitem{chen(2010)} Y. Yan and G. Y. Chen, OD-Characterization of alternating and
symmetric groups of degree 106 and 112, {\em Proceedings of the
International Conference on Algebra}, 2010, 690--696.

\bibitem{chen(2013)}
Y. Yan, G. Y. Chen, L. C. Zhang and H. Xu, Recognizing finite
groups through order and degree patterns, {\em Chin. Ann. Math.
Ser. B}, 34(2) (2013), 777--790.

\bibitem{L(3q)}  A. V. Zavarnitsine, Recognition of the simple groups $L_3(q)$ by element orders, {\em  J. Group Theory},  7 (2004), 81--97.

\bibitem{ZU(3q)}  A. V.
Zavarnitsin, Recognition of the simple groups ${\rm U}_3(q)$ by element orders, {\em  Algebra Logic},  45(2) (2006), 106--116.

\bibitem{za-primegraph} A. V. Zavarnitsine, Recognition of finite groups by the prime graph, {\em Algebra Logic}, 45(4) (2006),  220--231.

\bibitem{zav-L4} A. V. Zavarnitsine, Exceptional action of the simple groups ${\rm
L}\sb 4(q)$ in the defining characteristic, {\em Sib. \'Elektron.
Mat. Izv.}, 5 (2008), 68--74.

\bibitem{za}  A. V. Zavarnitsine, Finite simple groups with narrow prime spectrum,
{\em Sib. Elektron. Mat. Izv.}, 6 (2009), 1--12.

\bibitem{ZhangL65} L. C. Zhang, H. Lu, D. P. Yu and S. M. Chen, OD-Characterization of Finite Simple Group $L_6(5)$,  {\em Journal of Southwest University (Natural Science Edition)}, 33(12)(2012), 82--85.

\bibitem{ZS-2009} L. C. Zhang and W. J. Shi, OD-Characterization of simple $K_4$-groups,
{\em Algebra Colloq.}, 16 (2009), 275--282.

\bibitem{ZS-2010} L. C. Zhang and W. J. Shi, OD-Characterization of almost simple groups related to $U_3(5)$, Acta Math. Sin. (Engl. Ser.), 26 (2010), 161--168.

\bibitem{Zhang-Shi-2011} L. C. Zhang and W. J. Shi,
OD-characterization of almost simple groups related to $U\sb
6(2)$, {\em Acta Math. Sci. Ser. B Engl.} Ed., 31(2) (2011),
441--450.

\bibitem{ZS-2012} L. C. Zhang and W. J. Shi, OD-Characterization of the projective special linear
groups $L_2(q)$, {\em Algebra Colloq.}, 19(3)(2012), 509--524.

\bibitem{regular} L. C. Zhang, W. J. Shi,  Y. Dapeng and W. Jin, Recognition of finite simple groups whose first prime graph components are $r$-regular,  {\em Bull. Malays. Math. Sci. Soc.},  36(1) (2013), 131--142.

\bibitem{Zsigmondy} K. Zsigmondy, Zur Theorie der Potenzreste, {\em Monatsh. Math. Phys.}, 3 (1892),  265--284.
\end{thebibliography}
\end{document}